\documentclass[12pt, reqno]{amsart}
\usepackage{color}

\usepackage{amsmath}
\usepackage{amssymb}
\usepackage{amsfonts}
\usepackage{mathrsfs}
\usepackage{amsbsy}
\usepackage{relsize}
\usepackage[colorlinks, linkcolor=red, citecolor=blue, urlcolor=blue, pagebackref, hypertexnames=false]{hyperref}
\usepackage[hyperpageref]{backref}
\usepackage{enumitem}

\usepackage{tikz}

\numberwithin{equation}{section}
\newcounter{bbb}
\numberwithin{bbb}{section}

\newtheorem{theorem}[bbb]{Theorem}
\newtheorem{lemma}[bbb]{Lemma}
\newtheorem{proposition}[bbb]{Proposition}

\newtheorem{remark}[bbb]{Remark}

\newcommand{\R}{\mathbb R}

\newcommand{\ba}{\mathbf a}
\newcommand{\baa}{\mathbf A}
\newcommand{\bE}{\mathbf E}

\newcommand{\bV}{\mathbf V}
\newcommand{\bK}{\mathbf K}

\DeclareMathOperator*{\rad}{{\bf r}}
\DeclareMathOperator*{\rea}{Re}
\DeclareMathOperator*{\ima}{Im}

\title[Damped NLS]{Existence and nonexistence of global solutions  for time-dependent damped NLS equations}

\author[ M. Hamouda and M. Majdoub]{Makram Hamouda and Mohamed Majdoub}

\address[M. Hamouda]{Department of Basic Sciences, Deanship of Preparatory Year and Supporting Studies, Imam Abdulrahman Bin Faisal University, P. O. Box 1982, Dammam, Saudi Arabia.}
\email{\sl mmhamouda@iau.edu.sa }
\address[M. Majdoub]{Department of Mathematics, College of Science, Imam Abdulrahman Bin Faisal University, P. O. Box 1982, Dammam, Saudi Arabia \newline \& \newline Basic and Applied Scientific Research Center, Imam Abdulrahman Bin Faisal University, P.O. Box 1982, 31441, Dammam, Saudi Arabia}
\email{\sl mmajdoub@iau.edu.sa}
\email{\sl med.majdoub@gmail.com}
\subjclass[2020]{35Q55, 35A01, 35B30, 35B44.}
\keywords{Damped NLS equation, loss dissipation, global existence, blow-up.}

\begin{document}
\maketitle

\begin{abstract}
We investigate the Cauchy problem for the nonlinear Schr\"odinger equation with a time-dependent linear damping term. Under non standard assumptions on the loss dissipation, we prove the blow-up in the inter-critical regime, and the global existence in the energy subcritical case. Our results generalize and improve the ones in \cite{VDD, FZS14, OhTo}.
\end{abstract}


\section{Introduction and main results}
	\label{S1}

We deal with the initial value problem for the linearly time-dependent damped nonlinear  Schr\"odinger (NLS)
equation
\begin{equation}
\label{main}
{\rm i}\partial_t u+\Delta u +{\rm i} {\mathbf a}(t) u=\mu \,|u|^{p-1}u,\quad (t,x)\in (0,\infty)\times\R^N,
\end{equation}
with initial data $u(0,x) = u_0(x)$, where $p>1,\, \mu=\pm1$, and ${\mathbf a}: [0,\infty)\to [0,\infty)$ is a continuous function which represents the strength of loss dissipation. \\

Equation \eqref{main} can be seen as a model of dilute Bose-Einstein condensate (BEC) when the two-body interactions of the condensate are considered.  Also, it has been considered to study the laser guiding in an axially nonuniform plasma channel. One of the features of considering NLS with time-dependent coefficients is to control the nonlinear dynamics of the solution; this constitutes a main target in the understanding of the concept of optical soliton managements technology \cite{Malomed}. In the context of BEC, the generalized NLS equation with time-dependent coefficients reads as follows:
\begin{equation}
\label{GNLS}
{\rm i}\partial_t u+f(t)\Delta u +V(t,x)u+{\rm i} {\mathbf a}(t) u=g(t)|u|^{p-1}u,
\end{equation}
where $f(t)$, $V(t,x)$ and $g(t)$ represent the tunable effective mass, the trapping potential strengths and the Feshbach resonance management, respectively. The damping term $\ba(t)$ describes the loss ($\ba(t) \ge 0$) or the gain ($\ba(t) \le 0$) dissipation. For further readings about the physical context of \eqref{GNLS}, see e.g. \cite{CPKP06} and for a more general setting of the problem see the book \cite{Malomed}.\\

Before getting into the time-dependent damping context for \eqref{main}, we start by recalling  what has been done in the constant damping case. Note that the case ${\ba}=0$ in \eqref{main} corresponds to the standard NLS equation. It is worth to mention that there has been a large  amount of researches on the  NLS equation, and the monographs \cite{Cazenave, Tao} cover a very extensive overview on the most established results on the subject. Some interesting avenues in this direction have been  investigated for the $2D$ energy critical NLS equation \cite{CIMM}.  Furthermore, the case of inhomogeneous nonlinearities $\pm |x|^b\,|u|^{p-1}u\, (b\in\R)$ was  extensively studied, see e.g.  \cite{AT, DMS} and the references therein.\\

For the constant damping case, that is $\ba(t)=\ba> 0$ in \eqref{main}, many works have been devoted to the study of local/global well-posedness, blow-up and asymptotic behavior. 

For $\mu=1$ ({defocusing regime}), by using \eqref{M-Id} and \eqref{E-Id} below together with the standard blow-up alternative, one can easily check that the maximal $H^1-$solution is global provided that  $1<p<1+\frac{4}{N-2}$ ( $p<\infty$ if $N=1,2$); see \cite{Cazenave}. 

For $\mu=-1$ ({focusing regime}), by using \eqref{M-Id} below and the $L^2$ blow-up alternative \cite{Cazenave}, one can see that the maximal solution is global  if $1<p<1+\frac{4}{N}$. See also Proposition \ref{prop-appendixB} for the proof in the time-dependent case. Moreover,  it is shown in \cite{In} that the solution scatters. However, the scattering fails to hold for $\ba=0$ and $1<p<1+\frac{4}{N}$. For $1+\frac{4}{N}\leq p<1+\frac{4}{N-2}$  and $u_0\in H^1(\R^N)$, Ohta-Todorova \cite{OhTo} proved that the corresponding maximal solution is global for $\ba$ sufficiently large. Making use of Merle-Rapha\"el techniques \cite{Merle1}, the stability of finite time blow-up dynamics with the log-log blow-up speed is proven in \cite{Dar} for $p=1+\frac{4}{N}$.

The inter-critical regime, that is $1+\frac{4}{N}<p<1+\frac{4}{N-2}$ ($1+\frac{4}{N}<p<\infty$ if $N=1,2$), was first investigated by Tsutsumi in \cite{Tsut1}. It was shown that if $u_0\in \Sigma(\R^N):=H^1(\R^N)\cap L^2(\R^N, |x|^2dx)$ satisfies some suitable conditions, then the corresponding maximal solution to \eqref{main} blows up in finite time. This result was improved by Ohta and Todorova in \cite{OhTo} by allowing the energy of $u_0$ to be positive. In addition, Ohta-Todorova \cite{OhTo} consider some invariant sets for which the stemming solutions are global, see also \cite{CZW}.

The asymptotic behavior of global solutions to \eqref{main} with constant damping was recently investigated in \cite{VDD, In} for both focusing and defocusing regimes.\\

Now, we turn back to the Cauchy problem with  time-dependent damping $\ba(t)$. In this case, the local well-posedness in $H^1(\R^N)$ for  $1< p<1+\frac{4}{N-2}$ ($1< p<\infty$ if $N=1,2$)  is a straightforward consequence of  classical NLS arguments \cite{Cazenave, Kato}, see also Appendix \ref{appendix2} below.

We are mainly interested in global existence and blow-up of solutions to \eqref{main} with  initial data in $H^1(\R^N)$.  To the best of our knowledge,  there are a few works dealing with the global versus blow-up of the linearly time-dependent damped NLS equation. In \cite{FZS14}, the blow-up for \eqref{main}, in the inter-critical regime, is proven in $\Sigma(\R^N)$ for a non-decreasing damping by assuming the  smallness of  both $\|\ba\|_{L^{\infty}}$ and $\|\ba'\|_{L^{\infty}}$. Moreover, the global existence is shown in \cite{FZS14}  under some conditions on the loss/gain damping term $\ba(t)$.

In the present article, we generalize and somehow improve  earlier existing works \cite{VDD, FZS14, OhTo, Tsut1} regarding the constant and the time-dependent damping cases. One of the novelties here is the non standard assumptions on the time-dependent damping term which cover the hypotheses in earlier works \cite{FZS14, OhTo}.

Let us finally mention that some other interesting directions are  investigated by many authors; see e.g. \cite{Antonelli2} for the nonlinear damping and \cite{Dar2, Tarek} for the fractional dissipation, and the references therein.

Unless otherwise specified, along  the rest of this article, we will only consider the focusing case, that is $\mu =-1$.

\subsection*{Main results}
\label{SS1}

Before we state our main results, we first introduce the following notations:
\begin{gather}
\label{A}
\baa(t)=\int_0^t\,\ba(s)ds,\\
\label{ais}
\underline{\ba}=\inf_{t>0}\,\bigg(\frac{\baa(t)}{t}\bigg), \quad \overline{\ba}=\sup_{t>0}\,\bigg(\frac{\baa(t)}{t}\bigg).
\end{gather}
One can easily verify that if $\ba \in L^q(0, \infty)$ for some $1\le q\le \infty$, then  $\overline{\ba}<\infty$. However, the opposite is not in general true, see Lemma \ref{Explicit-Examp} in  Appendix \ref{appendix1}.

Let us denote the subspace of $H^1(\R^N)$ consisting of radial functions by
	\[
	H^1_{\rad}(\R^N) := \left\{f \in H^1(\R^N) \ : \ f(x)= f(|x|)\right\}.
	\]
Denote also by 
$$\Sigma(\R^N):=H^1(\R^N)\cap L^2(\R^N;|x|^2 dx).$$
The following blow-up result can be seen as an improvement of \cite[Theorem 1.4]{FZS14}.
\begin{theorem}
\label{Blow0}
Let $N\geq 2$,  $1+\frac{4}{N}< p<1+\frac{4}{N-2}$ ($p<\infty$ if $N=2$), and $u_0\in \Sigma(\R^N)$. Suppose that $\bE(u_0)<0$, $\bV(u_0)<0$, and $\ba'(t)\geq 0$. Then for 
\begin{equation}
\label{a-star}
\overline{\ba}<\ba_*:=\frac{-2\bV(u_0)}{(\kappa+1)\bK(u_0)},
\end{equation}
we have 
\begin{equation}
\label{Lspan}
    T_{\ba}^*(u_0)\leq \frac{-1}{2(\kappa+1)\overline{\ba}}\ln\left(1+\frac{(\kappa+1)
\overline{\ba}\bK(u_0)}{2\bV(u_0)}\right)<\infty,
\end{equation}
 where 
\begin{equation}
\label{kappa}
\kappa=\frac{N+2-(N-2)p}{N(p-1)-4},
\end{equation}
and $\bE, \bK, \bV$ are respectively given by \eqref{Ener}, \eqref{Ku} and \eqref{Vu} below. 
\end{theorem}

In view of the blow-up conclusions stated in the above theorem, some comments arise, we enumerate them in what follows.

\begin{itemize}
    \item[($i$)] Compared to \cite[Theorem 1.4]{FZS14}, one of the novelties in Theorem \ref{Blow0} consists in obtaining the blow-up without the smallness assumption on the derivative of $\ba(t)$ which seems unusual. Furthermore, we explicitly give an upper bound for the blow-up region, namely $[0, \ba_*)$ as in \eqref{a-star}, and an estimate of the lifespan as in \eqref{Lspan}.
    \item[($ii$)] It is worth mentioning that Theorem \ref{Blow0} covers the constant damping case, and somehow improves \cite[Theorem 1.2]{OhTo} in the sense that we give upper bounds for $\ba_*$  and $T_{\ba}^*(u_0)$. In contrast with the assumptions in \cite[Theorem 1.2]{OhTo}, the hypotheses in Theorem \ref{Blow0} are stronger with regard to the blow-up target. Nevertheless, our aim in Theorem \ref{Blow0} is to derive  estimates for $\ba_*$ and $T_{\ba}^*(u_0)$.
    \item[($iii$)]Clearly, one can easily see that in general $\overline{\ba}\leq \|\ba\|_{L^\infty}$. However, due to Lemma \ref{L-a}, we have the equality for a non-decreasing function $\ba$, namely $\overline{\ba}=\|\ba\|_{L^\infty}$.
    \item[($iv$)] The approach used to prove the blow-up result in Theorem \ref{Blow0} applies only in the inter-critical regime, that is $1+\frac{4}{N}< p<1+\frac{4}{N-2}$. This makes the constant $\kappa$ (given by \eqref{kappa}) finite and positive. The same aforementioned restriction on $p$ appears in \cite{FZS14, OhTo}.
    
   \item[($v$)] With the approach used in \cite{FZS14}, it is not clear how to remove the assumption $\ba'(t)\geq 0$ in Theorem \ref{Blow0}. The aforementioned hypothesis on the damping function $\ba(t)$ seems not natural. For example, the damping $t\mapsto (1+t)^{-\theta},~\theta>0$ is not covered by \cite[Theorem 1.4]{FZS14}. Note that this type of damping is extensively considered in the blow-up problems related to the nonlinear wave equation, see e.g. \cite{HH-MJM2022} for the scale-invariant case and \cite{LST-2020} for an overview on the other different cases (overdamping, effective, scale-invariant and scattering).
    \item[($vi$)] In order to get rid from the monotony argument for the damping function $\ba(t)$, we follow another approach which consists in rewriting equation \eqref{main} in terms of $v(t,x):={\rm e}^{\baa(t)}\,u(t,x)$, see e.g. \cite{VDD}. 
\end{itemize}

Thanks to Proposition \ref{Rela} and using some localized virial estimates \cite{VDD}, we obtain the following  blow-up result for the inter-critical regime.
\begin{theorem}
\label{Blow1}
Let {$N\geq 2$}  and $p \ge 1+\frac{4}{N}$ such that
$$\left\{\begin{array}{ll} p \le 1+\frac{4}{N-2} \quad &\text{if} \ \ N \ge 3,\\
p < \infty  \quad &\text{if} \ \ N=2.
\end{array}\right.$$
Assume that $u_0\in \Sigma(\R^N)$ satisfies one of the following conditions
\begin{itemize}
    \item[(I)] ${\mathbf E}(u_0)<0$,
    \item[(II)] ${\mathbf E}(u_0)=0$ and ${\mathbf V}(u_0)<0$,
    \item[(III)] ${\mathbf E}(u_0)>0$ and ${\mathbf V}(u_0)+\sqrt{{\mathbf E}(u_0){\mathbf K}(u_0)}<0$,
\end{itemize}
where $\bE, \bK, \bV$ are respectively given by \eqref{Ener}, \eqref{Ku} and \eqref{Vu} below.  Then, there exists $\ba_*>0$ such that for all  $\overline{\ba}\in [0,\ba_*)$, the corresponding maximal solution to \eqref{main} blows up in finite time.
\end{theorem}

Now, if we only assume that $u_0\in H^1_{\rad}(\R^N)$ then the blow-up holds for small damping $\overline{\ba}$ as stated below.
\begin{theorem}
\label{Blow2}
Let {$N\geq 2$}  and $p \ge 1+\frac{4}{N}$ such that
$$\left\{\begin{array}{ll} p \le 1+\frac{4}{N-2} \quad &\text{if} \ \ N \ge 3,\\
p \leq 5 \quad &\text{if} \ \ N=2.
\end{array}\right.$$

If $u_0\in H^1_{\rad}(\R^N)$ satisfies  ${\mathbf E}(u_0)<0$ ($\bE$ is given by \eqref{Ener} below), then there exists $\ba_*>0$ such that for all $\overline{\ba}\in [0,\ba_*)$, the corresponding maximal solution to \eqref{main} blows up in finite time.
\end{theorem}
\begin{remark} \rm
   Note that, in the energy subcritical regime, the constant $\ba_*$ in Theorems \ref{Blow1}, \ref{Blow2} depends on the size of the initial data, that is $\ba_*=\ba_*(\|u_0\|_{H^1})$. However, in the energy critical regime, the constant $\ba_*$ depends on the profile of $u_0$.
\end{remark}
\begin{remark}
~{\rm \begin{itemize}
\item[($i$)] Theorem \ref{Blow2} improves \cite[Lemma 3.7]{VDD} in the sense  that we consider here a time-dependent damping function while a constant damping is studied in \cite{VDD}. 
    \item[($ii$)] Obviously, the hypothesis in Theorem \ref{Blow2} is independent of the cut-off function $\varphi_R$ given by \eqref{varphi-R} below. However, using Lemma \ref{Blow-R} with $a_R:= 2N(p-1){\mathbf E}(u_0)+o_R(1)$, $b_R:=2\ima\int_{\R^N} \overline{u_0}\left(\nabla \varphi_R \cdot \nabla u_0\right) dx $ and $c_R:=\int_{\R^N}\, \varphi_R(x)|u_0(x)|^2 dx$, one can obtain the same result as in Theorem \ref{Blow2} under one of the following assumptions:    
    \begin{itemize}
    \item[$\bullet$] ${\mathbf E}(u_0)=0$ and $\displaystyle{\sup_{R\gg 1}}\,(b_R)<0$\footnote{The condition $\displaystyle{\sup_{R\gg 1}}\,(b_R)<0$ means that there exist $R_0, \delta_0>0$ such that $b_R\leq -\delta_0$ for all $R\geq R_0$.};
    \item[$\bullet$] ${\mathbf E}(u_0)>0$ and $b_R+2 \sqrt{a_Rc_R}<0$ for $R$ large enough.
\end{itemize}
 \item[($iii$)] It is not clear how to remove the radial symmetry assumption in Theorem \ref{Blow2}. More precisely, can we have blowing-up solutions to \eqref{main} (even for the undamped case, that is $\ba(t)=0$) under the following conditions: 
 \begin{equation}
    \label{Neg-Conj}
    u_0\in H^1(\R^N),\quad \mathbf{E}(u_0)<0\quad\text{and}\quad 1+\frac{4}{N}\leq p< \left\{\begin{array}{ll} 1+\frac{4}{N-2} & \text{if} \ N \ge 3\\ \infty  & \text{if} \ N =1,2\end{array} \right. ?
 \end{equation}
To the best of our knowledge, the only positive answer to the above question  is given in \cite{OT-1D} for $N=1$, $p=5$ and $\ba(t)=0$. In higher dimensions, a partial answer is proved in \cite{Merle} for the mass-critical case and under some additional assumptions on the initial data.
 
\end{itemize}}
\end{remark}

{{ We now turn to the global existence for \eqref{main}  in the energy space $H^1(\R^N)$. For the $L^2-$subcritical regime, that is $p<1+\frac{4}{N}$, the global existence follows from  the local well-posedness in $H^1(\R^N)$, \eqref{M-Id}-\eqref{E-Id} below,  and the following Gagliardo-Nirenberg inequality
\begin{equation}
\label{GNI}
\|f\|_{L^{p+1}}\, \lesssim\,\|f\|_{L^2}^{1-\frac{N(p-1)}{2(p+1)}}\, \|\nabla f\|_{L^2}^{\frac{N(p-1)}{2(p+1)}},\quad \text{for all} \ f\in H^1(\R^N).
\end{equation}

For the sake of completeness, we give in Appendix \ref{appendix2} the proof of the global existence of \eqref{main} in the $L^2-$subcritical regime ($p<1+\frac{4}{N}$).

Therefore, we only consider the case $p\geq 1+\frac{4}{N}$. Our main result reads as follows.
\begin{theorem}
\label{GE}
Let $N\geq 3$, $1+\frac{4}{N}\leq p<1+\frac{4}{N-2}$ and $u_0\in H^1(\R^N)$. Then, there
exists a constant $C>0$ independent of $u_0$ and $\ba$ such that $T^*_{\ba}(u_0)=\infty$ for all $\underline{\ba}\geq  C\|u_0\|_{H^1}^{\theta}$, where $\underline{\ba}$ is given by \eqref{ais}, and 
\begin{equation}
\label{theta}
\theta=\frac{2(p-1)(p+1)}{4-(N-2)(p-1)}.
\end{equation}
\end{theorem}
\begin{remark}
~{\rm \begin{itemize}
\item[($i$)] The above theorem improves \cite[Theorem 1]{OhTo} in two ways. First, it covers the constant damping setting since in this case we clearly have $\underline{\ba}=\ba$. Second, it gives an explicit lower bound for $\underline{\ba}$ in terms of $\|u_0\|_{H^1}$. 
\item[($ii$)] Among many others, the non constant damping $\ba(t)=\lambda(1-{\rm e}^{-t}), \lambda>0$, is obviously covered by Theorem \ref{GE}.
\item[($iii$)] The proof of Theorem \ref{GE} uses similar ideas as in \cite{OhTo} and the restriction $p<1+\frac{4}{N-2}$ naturally appears when using \eqref{theta}. We have no clue whether such restriction on $p$ could be removed.
\end{itemize}
}
\end{remark}
\begin{remark}
~{\rm \begin{itemize}
\item[($i$)] It is worth to mention that to prove Theorem \ref{GE} we assume that  $\underline{\ba}>0$ which in particular yields $\displaystyle\int_0^\infty \ba(t)dt=\infty$ (see \eqref{Glo-epsi} below). A natural question arises: what happens if $\underline{\ba}=0$? \item[($ii$)] To answer the above question, one can introduce a refined measurement of the damping function  $\ba(t)$. Therefore, instead of \eqref{ais} we define the {\it weighted average measurement} ({\bf WAM}) of the damping function by
\begin{equation}
\label{WAM}
 \underline{\ba}(\alpha, \beta, \gamma, \delta, \sigma):=\inf_{t>0}\,\left(\frac{1}{\mathbf{w}(t)}{\displaystyle\int_0^t \ba(s) ds}\right),
\end{equation}
where $\mathbf{w}(t)=t^\alpha\,(1+t)^{\beta}  (\ln(1+t))^{\gamma} {\rm e}^{\delta t^{\sigma}}$ with an adequate choice of $\alpha, \beta, \gamma, \delta, \sigma \in\R$. Obviously, we have  $\underline{\ba}(1, 0,0,0,0)=\underline{\ba}$ as in \eqref{ais}. Furthermore, for the damping $(1+t)^{-\theta}$, the {\bf WAM} can be chosen as follows:
$$
(\alpha, \beta, \gamma, \delta, \sigma)=\left\{\begin{array}{ll}   (0,0,1,0,0) \quad &\text{if} \ \ \theta=1,\\
 (1,-\theta,0,0,0) \quad &\text{if} \ \ \theta<1.
\end{array}\right.$$
\item[($iii$)] We believe that the global existence result of Theorem \ref{GE} can be extended to the class of damping functions that satisfy the largeness of $\underline{\ba}(\alpha, \beta, \gamma, \delta, \sigma)$ for suitable parameters $\alpha, \beta, \gamma, \delta, \sigma$.
\item[($iv$)] We think that the {\bf WAM} can be used in many other damped equations.
\end{itemize}
}
\end{remark}

\medskip

We conclude the introduction with an outline of the paper. In Section \ref{S2}, we introduce some useful notations and auxiliary results. Section \ref{S3} is devoted to the proofs of our main results. Finally, in Appendix \ref{appendix1}, we show the utility of the refinement on the new assumptions on the damping term. Appendix \ref{appendix2} is concerned with the global existence in the focusing $L^2-$subcritical case.

\section{Useful tools \& Auxiliary results}
\label{S2}
To sate our main results in a clear way, we define the following quantities:
\begin{gather}
\label{Mass}
{\mathbf M}(u(t))=\int_{\R^N}\,|u(t,x)|^2\,dx,
\vspace{.3cm}\\
\label{Ener}
{\mathbf E}(u(t))=\|\nabla u(t)\|_{L^2}^2-\frac{2}{p+1}\int_{\R^N}\, |u(t,x)|^{p+1}\,dx,\vspace{.3cm}\\
\label{Iu}
{\mathbf I}(u(t))=\|\nabla u(t)\|_{L^2}^2-\int_{\R^N}\, |u(t,x)|^{p+1}\,dx,\vspace{.3cm}\\
\label{Ku}
{\mathbf K}(u(t))=\int_{\R^N}\, |x|^2 |u(t,x)|^2\,dx,\vspace{.3cm}\\
\label{Vu}
{\mathbf V}(u(t))=\ima \Big(\int_{\R^N} \left(x\cdot \nabla u(t,x)\right) \overline{u}(t,x)\,dx\Big),\vspace{.3cm}\\
\label{Pu}
{\mathbf P}(u(t))=\|\nabla u(t)\|_{L^2}^2-\frac{N(p-1)}{2(p+1)}\int_{\R^N}\, |u(t,x)|^{p+1}\,dx.
\end{gather}

Some crucial relationships between the above quantities are summarized in the following.
\begin{proposition}
\label{Rela}
Let $u$ be a sufficiently smooth solution of \eqref{main} with initial data $u(0,x)=u_0(x)$ on $0\leq t\leq T$. Then, we have
\begin{gather}
\label{M-Id}
{\mathbf M}(u(t))={\rm e}^{-2\baa(t)}{\mathbf M}(u_0),\vspace{.3cm}\\
\label{E-Id}
\frac{d}{dt}{\mathbf E}(u(t))=-2\ba(t) {\mathbf I}(u(t)),\vspace{.3cm}\\
\label{K-Id}
\frac{d}{dt}{\mathbf K}(u(t))+2\ba(t) {\mathbf K}(u(t))=4{\mathbf V}(u(t)),\vspace{.3cm}\\
\label{V-Id}
\frac{d}{dt}{\mathbf V}(u(t))+2\ba(t) {\mathbf V}(u(t))=2{\mathbf P}(u(t)).
\end{gather}
\end{proposition}

\begin{proof}[{Proof of Proposition \ref{Rela}}]
Although we deal here with a time-dependent damping, the proof mimics the same steps performed in \cite[Lemma 1]{Tsut1} where the constant damping is considered. The details are hence omitted.
\end{proof}

In order to facilitate the reading of this paper, we collect here some mathematical tools. 
{The following continuity argument (or bootstrap argument) will also be useful for our purpose. 
\begin{lemma}\cite[Lemma 3.7, p. 437]{Strauss}\\
\label{boots}
Let $\mathbf{I}\subset\R$ be a time interval, and $\mathbf{X} : \mathbf{I}\to [0,\infty)$ be a continuous function satisfying, for every $t\in \mathbf{I}$,
\begin{equation*}
		\label{boots1}
		    	\mathbf{X}(t) \leq a + b [\mathbf{X}(t)]^\theta,
		\end{equation*}
	where $a,b>0$ and $\theta>1$ are constants. Assume that, for some $t_0\in \mathbf{I}$,
		\begin{equation*}
		\label{boots2}
	\mathbf{X}(t_0)\leq a, \quad a\,b^{\frac{1}{\theta-1}} <(\theta-1)\,\theta^{\frac{\theta}{1-\theta}}.
				\end{equation*}
		Then, for every $ t\in \mathbf{I}$, we have
		\begin{equation*}
		\label{boots3}
		    	\mathbf{X}(t)< \frac{\theta\,a}{\theta-1}.
		\end{equation*}
\end{lemma}

Using a change of unknowns $v(t,x)={\rm e}^{\baa(t)}u(t,x)$, where $\baa(t)$ is given by \eqref{A}, the equation \eqref{main} can be written as
\begin{equation}\label{main-bis}
\left\{
\begin{array}{rcl}
{\rm i}\partial_t v+\Delta v &=&-{\rm e}^{(1-p)\baa(t)}|v|^{p-1}v,\quad (t,x)\in [0,\infty)\times\R^N,\\
v(0,x)&=&u_0(x).
\end{array}
\right.
\end{equation}
Obviously, the blow-up or the global existence of $v$ implies those of $u$, and vice-versa.

The local well-posedness of \eqref{main-bis} follows as for the classical NLS since $t \mapsto {\rm e}^{(1-p)\baa(t)}$ is bounded, see e.g. \cite{Cazenave}. Furthermore, the mass conservation still holds true for $v$, namely
\begin{equation}\label{mass-v}
    \mathbf{M}(v(t))=\mathbf{M}(u_0).
\end{equation}
Now, we define the Hamiltonian of $v$ by
\begin{equation}\label{Hamil-v}
    \mathbf{H}(v(t)):=\|\nabla v(t)\|_{L^2}^2-\frac{2}{p+1}{\rm e}^{(1-p)\baa(t)}\|v(t)\|_{L^{p+1}}^{p+1}.
\end{equation}
A direct computation gives
\begin{equation}\label{Hamil-v-dt}
    \frac{d}{dt}\mathbf{H}(v(t))=\frac{2(p-1)}{p+1} \,\ba(t){\rm e}^{(1-p)\baa(t)}\|v(t)\|_{L^{p+1}}^{p+1}.
\end{equation}
Integrating the above equation in time yields
\begin{equation}\label{Hamil-v-id}
    \mathbf{H}(v(t))=\mathbf{E}(u_0)+\frac{2(p-1)}{p+1}\int_0^t  \ba(s){\rm e}^{(1-p)\baa(s)}\, \|v(s)\|_{L^{p+1}}^{p+1}\, ds.
\end{equation}

\begin{lemma} [Virial identity] \label{lem-viri-iden}
		Let $1< p\le \frac{N+2}{N-2}$ ($1< p<\infty$ if $N=1,2$). Let $\varphi: \R^N \rightarrow \R$ be a sufficiently smooth and decaying function. Let $v \in C([0,T), H^1(\R^N))$ be a solution to \eqref{main-bis}. Define
		\begin{align} \label{V-varphi}
		V_{\varphi}(t) := \int_{\R^N} \varphi(x)|v(t,x)|^2 dx.
		\end{align}
		Then, we have
		\[
		V'_{\varphi}(t) =  2 \ima \int_{\R^N} \overline{v}(t,x)\bigg(\nabla \varphi(x) \cdot \nabla v(t,x)\bigg) dx,
		\]
		and
		\begin{align}\label{V-second}
		V''_{\varphi}(t) &= - \int_{\R^N} \Delta^2 \varphi(x) |v(t,x)|^2 dx + 4 \rea \sum_{j,k=1}^N \int_{\R^N} \partial^2_{jk} \varphi(x) \partial_j\overline{v}(t,x) \partial_k v(t,x) dx \nonumber\\
		&\mathrel{\phantom{=}} - \frac{2(p-1)}{p+1} {\rm e}^{(1-p)\baa(t)}\int_{\R^N} \Delta \varphi(x)  |v(t,x)|^{p+1} dx.
		\end{align}
	\end{lemma}
 The proof of \eqref{V-second} is a straightforward application of  \cite[Lemma 5.3]{Tao07} for the nonlinearity $\mathcal{N}=-{\rm e}^{(1-p)\baa(t)}|v|^{p-1}v$. Note that if, in addition, the smooth function $\varphi$ and the solution $v$ are radial then
 \begin{equation}
    \label{phi-j-k}
    \sum_{j,k=1}^N \partial^2_{jk} \varphi\, \partial_j\overline{v} \,\partial_k v=\varphi''\,|\nabla v|^2.
 \end{equation}
 
Set $\varphi(x) = |x|^2$ in \eqref{V-varphi}. A straightforward computation yields the following virial identity.
\begin{lemma}\label{lem-viri-iden-x2}
Let $1< p\le \frac{N+2}{N-2}$ ($1< p<\infty$ if $N=1,2$). For $u_0 \in \Sigma(\R^N)$, the maximal solution to \eqref{main-bis} satisfies $v \in C([0,T^*), \Sigma(\R^N))$ and
    \begin{align} \label{viri-iden-x2}
	\frac{d^2}{dt^2}\int_{\R^N} |x|^2 |v(t,x)|^2dx=8 \|\nabla v(t)\|_{L^2}^2-{\frac{4N(p-1)}{p+1}}{\rm e}^{(1-p)\baa(t)}\int_{\R^N} |v(t,x)|^{p+1}dx.
	\end{align}
\end{lemma}

In order to prove our blow-up results for radially symmetric initial data,  we need some localized virial estimates. For that purpose, we introduce 
\[
	\Theta(r):= \int_0^r  \zeta(s)ds,
	\]
where the smooth  function $\zeta: [0,\infty) \rightarrow [0,\infty)$ is given by\footnote{The conditions on the behavior of the function $\zeta$ on $1<r<2$ is only necessary in the $L^2$-critical regime. For the inter-critical regime, it suffices  to consider a smooth nonnegative function $\zeta$ satisfying $\zeta(r)=2r$ for $0\leq r\leq 1$ and $\zeta(r)=0$ for $r\geq 2$.}
	\begin{equation}
 \label{zeta}
	\zeta(r)=
	\left\{
	\renewcommand*{\arraystretch}{1.2}
	\begin{array}{c c c}
	2r &\text{if} & 0 \leq r\leq 1,\\
	2[r-(r-1)^5] &\text{if} & 1<r\leq 1+\frac{1}{\sqrt[4]{5}},\\
	\zeta'(r)<0 &\text{if} & 1+\frac{1}{\sqrt[4]{5}}<r<2,\\
	0 &\text{if}  & r\geq2.
	\end{array}
	\right.
	\end{equation}

	For $R>0$, we define the following radial cut-off function
	\begin{align} \label{varphi-R}
	\varphi_R(x) = \varphi_R(r):= R^2\,\Theta\left(\frac{r}{R}\right), \quad r=|x|.
	\end{align}
	One can easily verify that, for all $r=|x|$,
	\begin{align}
 \label{phi-prop}
	\varphi''_R(r) \in [0,2], \quad \varphi'_R(r) \leq 2r,  \quad \Delta \varphi_R(x) \leq 2N.
	\end{align}

 \begin{lemma} \label{lem-viri-est-inte}
		Let $N\geq 2$ and $p\le \min \left(5,\frac{N+2}{N-2}\right)$. Let $u_0 \in  H^1_{\rad}(\R^N)$ and $v \in C([0,T^*), H^1_{\rad}(\R^N))$ be the maximal solution to \eqref{main-bis}. Let $\varphi_R$ be as in \eqref{varphi-R} and define $V_{\varphi_R}(t)$ as in \eqref{V-varphi}. Then for any $R>0$, we have for all $t\in [0,T^*)$,
		\begin{align} \label{viri-est-inte}
		\begin{aligned}
		V''_{\varphi_R}(t) \leq 8 \|\nabla v(t)\|^2_{L^2} &- {\frac{4N(p-1)}{p+1}}{\rm e}^{(1-p)\baa(t)}\int_{\R^N} |v(t,x)|^{p+1}dx \\
		&+ CR^{-2} + \left\{
		\begin{array}{ccc}
		C R^{-2(N-1)}\|\nabla v(t)\|^2_{L^2} &\text{if} & p=5, \\
		C R^{-\frac{(N-1)(p-1)}{2}} \left(\|\nabla v(t)\|^2_{L^2} + 1\right) &\text{if}& p<5,
		\end{array}
		\right.
		\end{aligned}
		\end{align}
		where $C$ is a positive constant.
	\end{lemma}
\begin{remark}
~{\rm
    \begin{itemize}
        \item[($i$)] The proof of Lemma \ref{lem-viri-est-inte} can be obtained by using  \eqref{V-second} with $\varphi = \varphi_R$ and by following the proof of \cite[Lemma 5.4]{DMS}.
        \item[($ii$)] Note that $$\min \left(5,\frac{N+2}{N-2}\right)=\left\{
		\begin{array}{ccc}
		 \displaystyle \frac{N+2}{N-2}&\text{if} & N \ge 3, \\
		5 &\text{if}& N=2,
		\end{array}
		\right.$$
  with the convention that $\frac{N+2}{N-2}=\infty$ if $N=2$. Consequently, the restriction on $p \le 5$ is only relevant in  dimension two, which is needed when applying the Young inequality. Removing this restriction in the two-dimensional case is an interesting open problem. 
        \item[($iii$)] Employing \eqref{viri-est-inte} together with \eqref{Hamil-v} and \eqref{Hamil-v-id}, we infer that
        \begin{align} \label{viri-est-consq}
		\begin{aligned}
		V''_{\varphi_R}(t) \leq& \ 2N(p-1) E(u_0)+ {\frac{4N(p-1)^2}{p+1}}\int_0^t  \ba(s){\rm e}^{(1-p)\baa(s)}\, \|v(s)\|_{L^{p+1}}^{p+1}\, ds
   \\& +2(4-N(p-1)) \|\nabla v(t)\|^2_{L^2}+ CR^{-2}\\
		& + \left\{
		\begin{array}{ccc}
		C R^{-2(N-1)}\|\nabla v(t)\|^2_{L^2} &\text{if} & p=5, \\
		C R^{-\frac{(N-1)(p-1)}{2}} \left(\|\nabla v(t)\|^2_{L^2} + 1\right) &\text{if}& p<5,
		\end{array}
		\right.
		\end{aligned}
		\end{align}
		where $C$ is a positive constant.
  \item[($iv$)] As we will see later, the inequality  \eqref{viri-est-consq} is only useful when $4-N(p-1) <0$, that is $p$ is $L^2-$supercritical.  However, when $p$ is $L^2-$critical, that is $4-N(p-1) =0$, we need a refined version of \eqref{viri-est-inte} as we will state in Lemma \ref{lem-viri-est-mass} below.
  
    \end{itemize}}
\end{remark}

 \begin{lemma}\label{lem-viri-est-mass}
		Let $N\geq 2$ and $p=1+\frac{4}{N}$. Let $u_0 \in  H^1_{\rad}(\R^N)$ and $v \in C([0,T^*), H^1_{\rad}(\R^N))$ be the maximal solution to \eqref{main-bis}. Let $\varphi_R$ be as in \eqref{varphi-R} and define $V_{\varphi_R}(t)$ as in \eqref{V-varphi}. Then for any $R, \varepsilon>0$, we have for all $t\in [0,T^*)$,
		\begin{align} \label{viri-est-mass}
		\begin{aligned}
		V_{\varphi_R}''(t) \leq 8\mathbf{H}(v(t))&+CR^{-2}+C \varepsilon R^{-2} + C\varepsilon^{-\frac{1}{N-1}}R^{-2} \\
		&-\int_{|x|>R}\Big(4\varphi_{1,R}(r)-\varepsilon(\varphi_{2,R}(r))^{\frac{N}{2}}\Big)|\nabla v(t,x)|^2dx,
		\end{aligned}
		\end{align}
		where $\mathbf{H}$ is given by \eqref{Hamil-v}, $C$ is a positive constant, and
		\begin{align} \label{psi-12-R}
		\varphi_{1,R}(r):=2-\varphi_R''(r),\quad \varphi_{2,R}(x):=2N-\Delta\varphi_R(x).
		\end{align}
	\end{lemma}
\begin{remark}
{\rm The proof of Lemma \ref{lem-viri-est-mass} follows line by line the one of Lemma 5.1 in \cite{DMS} (with $b=0$). Note that the  term coming from the time-dependent damping is controlled by $1$, that is ${\rm e}^{(1-p)\baa(t)}\leq 1$, for all $t \ge 0$. As in \cite[Lemma 5.1]{DMS}, the estimates $\|\Delta^2\varphi_R\|_{L^\infty}\lesssim R^{-1}$ and $\Big\| \nabla \Big(\varphi_{2,R}^{\frac{N}{4}}\Big) \Big\|_{L^\infty} \lesssim R^{-1}$ together with the $\varepsilon-$Young inequality constitute the main tools to conclude Lemma \ref{lem-viri-est-mass}.}
\end{remark}

\begin{remark}
{\rm For $\varepsilon>0$ small enough and using the explicit form of $\zeta$ as in \eqref{zeta}, we can see that
$$
4\varphi_{1,R}(r)-\varepsilon\left(\varphi_{2,R}(r)\right)^{\frac{N}{2}}\geq 0, \quad \forall\; r>R>0,
$$
and consequently we infer that
\begin{equation} \label{viri-est-mass-bis}
		V_{\varphi_R}''(t) \leq 8\mathbf{E}(u_0)+\frac{32}{N+2}\int_0^t  \ba(s){\rm e}^{-\frac{4}{N}\baa(s)}\, \|v(s)\|_{L^{2+4/N}}^{2+4/N}\, ds+ C\varepsilon^{-\frac{1}{N-1}}R^{-2},
		\end{equation}
  where we have used \eqref{Hamil-v-id}.
}
\end{remark}

\section{Proofs of main results}
\label{S3}
In this section, we will give the proofs of our main results related to the blow-up  and the global existence of the solution to \eqref{main}. The  proofs of blow-up results will be carried out by contradiction, that is we assume that the solution of \eqref{main} exists globally in time. 

Aiming to simplify our presentation, we denote by ${\mathbf M}(t)={\mathbf M}(u(t))$, ${\mathbf E}(t)={\mathbf E}(u(t))$, ${\mathbf K}(t)={\mathbf K}(u(t))$, ${\mathbf V}(t)={\mathbf V}(u(t))$ and ${\mathbf P}(t)={\mathbf P}(u(t))$.
\begin{proof}[{Proof of Theorem \ref{Blow0}}]
The proof is similar to the one in \cite[Theorem 1.4]{FZS14} except for the use of (4.12)-(4.13) in \cite{FZS14}. With our notations, and using (4.12)-(4.13) in \cite{FZS14} together with the assumption ${\mathbf E}(0)<0$, we end up with the following inequality
\begin{equation}
\label{FZS-In}
0\leq g(t){\rm e}^{-2\kappa \baa(t)}\, {\mathbf K}(t)\leq {\mathbf K}(0)g(t)+4{\mathbf V}(0)\int_{0}^t\,g(s) ds, 
\end{equation}
where $\kappa$, $\baa(t)$, ${\mathbf K}(t)$ and  ${\mathbf V}(t)$ are given, respectively, by \eqref{kappa}, \eqref{A}, \eqref{Ku} and \eqref{Vu}, and
\begin{equation}
\label{g}
g(t):= {\rm e}^{2(1+\kappa) \baa(t)}.
\end{equation}
Keeping in mind that ${\mathbf V}(0)<0$ and owing to \eqref{FZS-In}, we get
\begin{equation}
\label{FZS-C1}
\int_{0}^t\,{\rm e}^{2(1+\kappa) (\baa(s)-\baa(t))}\,ds= \frac{1}{g(t)}\int_{0}^t\,g(s) ds\leq \frac{-{\mathbf K}(0)}{4{\mathbf V}(0)}.
\end{equation}
Since $\ba$ is non-decreasing, then by Lemma \ref{L-a} we have 
$$
\baa(s)-\baa(t)\geq -\overline{\ba}(t-s),\quad \mbox{for all}\;\;\; 0\leq s\leq t.
$$
This leads to 
\begin{equation}
\label{FZS-C2}
\frac{1-{\rm e}^{-2(1+\kappa) \overline{\ba}\,t}}{2(1+\kappa) \overline{\ba}}\leq \frac{-{\mathbf K}(0)}{4{\mathbf V}(0)}.
\end{equation}
From \eqref{FZS-C2} we easily deduce \eqref{a-star} and \eqref{Lspan}. This finishes the proof of Theorem \ref{Blow0}.
\end{proof}

\begin{proof}[{Proof of Theorem \ref{Blow1}}]
    Using Lemma \ref{lem-viri-iden-x2}, \eqref{Hamil-v} and \eqref{Hamil-v-id}, we infer that
    \begin{eqnarray}
        \mathbf{K}''(t) &\le& 8\mathbf{E}(u_0)+\frac{16(p-1)}{p+1}\int_0^t  \ba(s){\rm e}^{(1-p)\baa(s)}\, \|v(s)\|_{L^{p+1}}^{p+1}\, ds \nonumber\\ &&+ \frac{4}{p+1} \left(4-N(p-1)\right) {\rm e}^{(1-p)\baa(t)}\, \|v(t)\|_{L^{p+1}}^{p+1} \nonumber\\
        &{\le}& 8\mathbf{E}(u_0)+\frac{16(p-1)}{p+1}\int_0^t  \ba(s){\rm e}^{(1-p)\baa(s)}\, \|v(s)\|_{L^{p+1}}^{p+1}\, ds,
        \label{K-second}
    \end{eqnarray}
    where we have used the fact that $4-N(p-1) \le 0$.\\
 Integrating \eqref{K-second} twice in time, we obtain
 \begin{equation}\label{K-t}
     0 \le \mathbf{K}(t) \le P(t)+\frac{16(p-1)}{p+1}Q(t),
 \end{equation}
 where 
 \begin{eqnarray}
     P(t)&=& 4 \mathbf{E}(u_0) t^2+ 4 \mathbf{V}(u_0) t + \mathbf{K}(u_0), \label{P-t}\\
     Q(t)&=&\int_0^t \int_0^\sigma\int_0^\tau \ba(s){\rm e}^{(1-p)\baa(s)}\, \|v(s)\|_{L^{p+1}}^{p+1}\, ds d\tau d\sigma. \label{Q-t}
 \end{eqnarray}
 From the assumptions $(I)$--$(III)$ in Theorem \ref{Blow1}, there exists $t_0>0$ such that $P(t_0)<0$. \\
 By Sobolev embedding, we have
 \begin{equation*}
     \sup_{0 \le s \le t_0} \|v(s)\|_{L^{p+1}}^{p+1} \lesssim \sup_{0 \le s \le t_0} \|v(s)\|_{H^1}^{p+1}:=C(t_0).
 \end{equation*}
 Hence, we deduce that
 \begin{equation}\label{Q-t0-1}
     Q(t_0) \le \frac{C(t_0)}{p-1}\int_0^{t_0} \int_0^\sigma \left(1-{\rm e}^{(1-p)\baa(\tau)}\right)\,  d\tau d\sigma.
 \end{equation}
 Thanks to \eqref{ais} and the elementary inequality $1-{\rm e}^{-x}\le x,$ for all $x\geq 0$, we infer that
 \begin{equation}\label{Q-t0-2}
     Q(t_0) \le \frac{t_0^3}{6}\, C(t_0)\, \overline{\ba},
 \end{equation}
 where $\overline{\ba}$ is given by \eqref{ais}.\\
 Combining \eqref{Q-t0-2} with the fact that $P(t_0)<0$ and using \eqref{K-t}, we infer the existence of $\ba_*>0$ such that $\mathbf{K}(t_0)<0$ for all $\overline{\ba} < \ba_*$, which is a contradiction.

 This concludes the proof of Theorem \ref{Blow1}.
\end{proof}
Now, we give the proof of Theorem \ref{Blow2} which is related to the blow-up of solutions to \eqref{main} in the radial case.
\begin{proof}[{Proof of Theorem \ref{Blow2}}]
First, we start by the inter-critical regime, that is $1+\frac{4}{N}<p<1+\frac{4}{N-2}$. For that purpose, we use \eqref{viri-est-consq} and the fact that $4-N(p-1)<0$ to obtain 
\begin{equation*}
V''_{\varphi_R}(t) \leq \ 2N(p-1) {\mathbf E}(u_0)+ o_R(1)+ {\frac{4N(p-1)^2}{p+1}}Q''(t),
\end{equation*}
where $Q(t)$ is defined by \eqref{Q-t}. 

Integrating the above inequality twice in time, we get
\begin{equation}
\label{V-eq-inter}
0\leq V_{\varphi_R}(t) \leq a_R t^2+b_R t+ c_R+{\frac{4N(p-1)^2}{p+1}}Q(t),
\end{equation}
where 
\begin{eqnarray*}
a_R&=& N(p-1) {\mathbf E}(u_0)+ o_R(1),\\
b_R&=&2\ima\int_{\R^N} \overline{u_0}\left(\nabla \varphi_R \cdot \nabla u_0\right) dx,\\ 
c_R&=&\int_{\R^N}\, \varphi_R(x)|u_0(x)|^2 dx.
\end{eqnarray*}
From Lemma \ref{Blow-R} and the assumption ${\mathbf E}(u_0)<0$, we deduce the existence of $t_0>0$ and $R_0>0$ such that 
$$
a_{R_0} t_0^2+b_{R_0} t_0+ c_{R_0}<0. 
$$
Combining the above information with \eqref{Q-t0-2}, we obtain a contradiction for small values of $\overline{\ba}$. 

For the $L^2-$critical case, that is $p=1+\frac{4}{N}$, by using \eqref{viri-est-mass-bis} (instead of \eqref{viri-est-consq}), the proof follows similarly as in the inter-critical regime.
\end{proof}
{{Finally, we give the proof of Theorem \ref{GE}.	As we will see in the subsequent, our proof  borrows some arguments from \cite{OhTo}.
\begin{proof}[{Proof of Theorem \ref{GE}}]
Denote by $U_{\ba}(t)$ the free propagator associated with \eqref{main}. One can easily verify that
\begin{equation}
\label{U-a}
U_{\ba}(t)={\rm e}^{-\baa(t)}\,{\rm e}^{it\Delta},
\end{equation}
where $\baa(t)$ is given by \eqref{A}. It is then quite classical that the Cauchy problem for \eqref{main} with initial data $u(0)=u_0$ can be written in an integral form (see \cite{Cazenave}):
\begin{equation}
\label{Integ-Eq}
u(t)=U_{\ba}(t)u_0-i\mu \int_0^t\,U_{\ba}(t-\tau)|u(\tau)|^{p-1}u(\tau)\,d\tau.
\end{equation}
Since ${\rm e}^{-\baa(t)}\leq 1$ for all $t\geq 0$, Proposition 3 in \cite{OhTo} remains true, and we have the following statement.
\begin{proposition}
\label{Oh-To}
Let $N\geq 3$, $1+\frac{4}{N}\leq p<1+\frac{4}{N-2}$ and $u_0\in H^1(\R^N)$. Then, there exists $\varepsilon>0$ independent of $\ba$ and $u_0$ such that $T^*_{\ba}(u_0)=\infty$ provided that $\|U_{\ba}(\cdot)u_0\|_{L^\theta(0,\infty; L^{p+1})}\leq \varepsilon$, where $\theta$ is given by \eqref{theta}.
\end{proposition}
The proof of Proposition \ref{Oh-To} follows the same lines as in \cite{OhTo} by using Strichartz estimates together with Lemma \ref{boots}. 

Now, to conclude the proof of the global existence it remains to show that we can make $\|U_{\ba}(\cdot)u_0\|_{L^\theta(0,\infty; L^{p+1})}$ small enough for large $\underline{\ba}$. To this end, we use a Sobolev embedding and \eqref{ais}, and we obtain
\begin{eqnarray}
\nonumber
\|U_{\ba}(\cdot)u_0\|_{L^\theta(0,\infty; L^{p+1})}^\theta&=&\int_0^\infty\,{\rm e}^{-\theta\baa(t)}\,\|{\rm e}^{it\Delta}u_0\|_{L^{p+1}}^{\theta}\,dt\\
\label{Glo-epsi}
&\lesssim& \|u_0\|_{H^{1}}^{\theta}\,\int_0^\infty\, {\rm e}^{-\theta {\underline{\ba}} t}\,dt\\
\nonumber
&\lesssim& \frac{\|u_0\|_{H^{1}}^{\theta}}{\theta \,{\underline{\ba}}}.
\end{eqnarray}
From \eqref{Glo-epsi} and Proposition \ref{Oh-To} we easily deduce that $T^*_{\ba}(u_0)=\infty$ if $\underline{\ba}\gtrsim \|u_0\|_{H^{1}}^{\theta}$. This finishes the proof of Theorem \ref{GE}.
\end{proof}
}}
\section{Concluding remarks}
\label{S5}
	
\begin{itemize}
    \item In the constant damping case, we know from \cite{OhTo} that there exist $a_*$ and $a^*$ such that $T_a^*(u_0) < \infty$ for $a < a_*$, and $T_a^*(u_0) = \infty$ for $a \ge a^*$. It will be interesting to investigate the sharpness of the upper bound of the blow-up region, namely $$\overline{a}_*:=\sup \left\{ a_*>0 \ \ \text{s.t.} \ T_a^*(u_0)< \infty \ \text{for all } \ 0<a< a_*\right\},$$  and the lower bound of the global existence region, that is $$\underline{a}^*:=\inf \left\{ a^*>0 \ \ \text{s.t.} \ T_a^*(u_0) = \infty \ \text{for all } \ a \ge a^*\right\}.$$
    Clearly, we have $0<\overline{a}_* \le \underline{a}^* < \infty$. 
    \item An interesting question is to see whether we have $\overline{a}_* = \underline{a}^*$? Otherwise, what happens in the threshold region $[\overline{a}_*, \underline{a}^*]$?
    \item We believe that, under some suitable conditions on the damping function $\ba(t)$, one can prove the global existence and the scattering in the energy-critical case for both focusing and defocusing regimes. This will be the subject of a forthcoming work; see \cite[Theorem 1.1]{VDD} for the case of a constant damping. 
    \item 
    In the mass-critical regime, that is $\displaystyle p=1+\frac{4}{N}$, and for the time-dependent damping, a curious question would be to show whether we can prove a similar result to \cite[Theorem 10]{OhTo}. This will be studied elsewhere.
\end{itemize}

\appendix
\section{}\label{appendix1}
\begin{lemma}
\label{L-a}
Let $\ba : [0,\infty)\to[0,\infty)$ be a continuous non-decreasing function. Then
\begin{equation*}
\label{L-aa}
\lim_{t\to\infty}\,\left(\frac{1}{t}\int_0^t\,\ba(s)ds\right)=\|\ba\|_{L^\infty}~(\in [0,\infty]).
\end{equation*}
In particular, we have $\overline{\ba}=\|\ba\|_{L^\infty}$, where $\overline{\ba}$ is given by \eqref{ais}.
\end{lemma}
\begin{proof}
Let $t>\tau>0$. Since $\ba(t)$ is non-decreasing, we have
\begin{eqnarray*}
\frac{1}{t}\int_0^t\,\ba(s)ds&=&\frac{1}{t}\int_0^\tau\,\ba(s)ds+\frac{1}{t}\int_\tau^t\,\ba(s)ds\\
&\geq&\frac{1}{t}\int_0^\tau\,\ba(s)ds+\frac{t-\tau}{t}\ba(\tau).
\end{eqnarray*}
Taking the limit-inf as $t$ goes to $\infty$, we obtain
$$
\liminf\limits_{t\rightarrow\infty}\left(\frac{1}{t}\int_0^t\,\ba(s)ds\right)\geq \ba(\tau), \quad \mbox{for all}\;\; \tau>0.
$$
Again, using the monotony of $\ba(t)$, we obtain by taking the limit as $\tau$ goes to $\infty$,
$$
\liminf\limits_{t\rightarrow\infty}\left(\frac{1}{t}\int_0^t\,\ba(s)ds\right)\geq \|\ba\|_{L^\infty}.
$$
We conclude the proof by observing that 
$$
\limsup\limits_{t\rightarrow\infty}\left(\frac{1}{t}\int_0^t\,\ba(s)ds\right)\leq \|\ba\|_{L^\infty}.
$$
\end{proof}
\begin{lemma}
\label{L-ab}
Let $\ba : [0,\infty)\to[0,\infty)$ be a continuous non-increasing function. Then
\begin{equation*}
\label{L-aaa}
\sup_{t>0}\,\left(\frac{1}{t}\int_0^t\,\ba(s)ds\right)=\ba(0)=\|\ba\|_{L^\infty}.
\end{equation*}
\end{lemma}
\begin{proof}
Since $\ba(t)$ is non-increasing and non-negative, then  $\ba(0)=\|\ba\|_{L^\infty}$. One can easily see that $\displaystyle\frac{1}{t}\int_0^t\,\ba(s)ds\leq \ba(0)$, for all $t>0$. Thanks to the continuity of $\ba(t)$, we infer that
$$
\lim_{t\to 0}\,\left(\frac{1}{t}\int_0^t\,\ba(s)ds\right)=\ba(0).
$$
This finishes the proof.
\end{proof}
\begin{lemma}
\label{Example}
Let $\alpha>0$. There exists a  continuous function $h_\alpha : [0,\infty)\to [0,\infty)$ such that
\begin{gather}
\label{Examp1}
h_\alpha\not\in L^\infty(0,\infty),\\
\label{Examp2}
\int_n^{n+1}\, |h_\alpha(t)|^q ~ dt=C_q\,n^{q-\alpha-1},\quad \forall\; n\geq 1,\;\;\;\forall\; q\in[1,\infty),
\end{gather}
where $C_q>0$ is a constant depending only on $q$.
\end{lemma}
\begin{proof}
For any integer $n\geq 1$, we define $h_\alpha$ on $[n, n+1]$ by
\begin{eqnarray*}
 h_\alpha(t)&=&\; \left\{
\begin{array}{cllll}16 n^{\alpha+2}\left(t-n\right)
\quad&\mbox{if}&\quad n\leq t\leq n+\frac{1}{4n^{\alpha+1}},\\\\ 16 n^{\alpha+2}\bigg(n+\frac{1}{2n^{\alpha+1}}-t\bigg) \quad
&\mbox{if}&\quad n+\frac{1}{4n^{\alpha+1}}\leq t\leq n+\frac{1}{2n^{\alpha+1}},\\\\
0 \quad
&\mbox{if}&\quad n+\frac{1}{2n^{\alpha+1}}\leq t\leq n+1.
\end{array}
\right.
\end{eqnarray*}
On $[0,1]$, we take $h_\alpha(t)=0$. One can easily verify that $h_{\alpha}$ satisfies \eqref{Examp1} and \eqref{Examp2} with $C_q=\frac{2^{2q-1}}{q+1}$.
\end{proof}
\begin{lemma}
\label{Explicit-Examp}
There exists a  continuous damping function $\ba : [0,\infty)\to [0,\infty)$ such that
\begin{gather*}
\label{Examp1-1}
\ba\not\in L^q(0,\infty),\quad \mbox{for all}~ ~\;1\leq q\leq \infty,\\
\label{Examp3-1}
\sup_{t>0}\bigg(\frac{1}{t}\int_0^t\,\ba(s)ds\bigg)<\infty.
\end{gather*}
\end{lemma}
\begin{proof}
The proof of Lemma \ref{Explicit-Examp} easily follows from Lemma \ref{Example} by taking $\alpha=1$ for instance. Indeed, let $\ba(t)=h_1(t)$ be the function given in Lemma \ref{Example} with $\alpha=1$. For $t>1$ there exists a (unique) integer $N\geq 1$ such that $N\leq t<N+1$. Hence, using \eqref{Examp2}, we have
\begin{eqnarray*}
\int_0^t\,\ba(s) ds=\int_1^t\,\ba(s) ds &\leq& \int_1^{N+1}\,\ba(s) ds \le t.
\end{eqnarray*}
Since $\displaystyle\int_0^t\,\ba(s) ds=0$ for all $0\leq t\leq 1$, we deduce that
$$
\sup_{t>0}\bigg(\frac{1}{t}\int_0^t\,\ba(s)ds\bigg)\leq 1<\infty.
$$
Next, for $1\leq q<\infty$, we have
\begin{equation*}
\int_0^\infty\,|\ba(t)|^q\,dt = \sum_{n=1}^\infty\,\int_n^{n+1}\,|\ba(t)|^q\, dt =C_q\sum_{n=1}^\infty\,  n^{q-2}=\infty.
\end{equation*}
This ends the proof of Lemma \ref{Explicit-Examp}.
\end{proof}
The following lemma, of which the proof is elementary, was used to conclude the blow-up in the radial case.
\begin{lemma}
\label{Blow-R}
For $R>0$, let $(a_R, b_R, c_R)\in \R^2\times (0,\infty)$ such that 
$$
\lim_{R\to\infty}\,a_R=a\in\R.
$$
Suppose one of the following conditions:
\begin{itemize}
    \item $a<0$;
    \item $a=0$, $\displaystyle{\sup_{R\gg 1}}\,(b_R)<0$ and $ c_R\lesssim 1$;
    \item $a>0$ and $b_R+2 \sqrt{a_Rc_R}<0$ for $R$ large enough.
\end{itemize}
Then, there exist $t_0>0$ and $R_0>0$ such that 
$$a_{R_0} t_0^2+b_{R_0} t_0+c_{R_0}  <0.$$
\end{lemma}

\section{}\label{appendix2}
We prove in this appendix the global existence in the energy space for  \eqref{main} when $p<1+\frac4N$. It is worth to mention that for the defocusing case ($\mu=1$), the global existence holds in $H^1(\R^N)$ for all $1<p < \frac{N+2}{N-2}$ when $N\ge 3$ and for all $p>1$ when $N=2$.
\begin{proposition}
    \label{prop-appendixB}
    Let $N\ge 2$, $1<p < 1+\frac4N$, and $\mu=-1$. Assume that $\ba \in L^1_{loc}(0,\infty)$ is a nonnegative function. Then, for any $u_0 \in H^1(\R^N)$, there exists a unique global solution $u \in C([0,\infty);H^1(\R^N))$ to \eqref{main} with $u(0,x)=u_0(x)$.
\end{proposition}
\begin{proof}
    The local well-posedeness in $H^1(\R^N)$  of \eqref{main} is a straightforward use of the methods in \cite{Cazenave} for \eqref{main-bis} and the boundedness of $t \mapsto {\rm e}^{(1-p)\baa(t)}$ on $[0,\infty)$. Specifically, there exits $0<T_{max}\le\infty$ and a unique maximal solution $u \in C([0,T_{max});H^1(\R^N))$ to \eqref{main} with $u(0,x)=u_0(x)$. Furthermore,  if $T_{max}<\infty$, then $\displaystyle\lim_{t \to T_{max}} \|\nabla u(t)\|_{L^2}=\infty$.

    Now, we argue by contradiction by assuming that $T_{max}<\infty$. For that purpose, we first integrate \eqref{E-Id} on $[0,t)$, and using $\ba(t) \ge 0$, we infer that
    \begin{equation}
        \label{B-energy-est1}
        \mathbf{E}(t) \le \|\nabla u_0\|^2_{L^2} + 2 \int_0^t \ba(s) \left( \int_{\R^N} |u(s,x)|^{p+1} dx\right) ds.
    \end{equation}
    Employing the Gagliardo-Nirenberg inequality \eqref{GNI}, Young's inequality and the fact that $\frac{N(p-1)}{2} <2$, we deduce that
    \begin{equation}
        \label{B-energy-est2}
         \int_{\R^N} |u(s,x)|^{p+1} dx \leq \frac12 \|\nabla u(s)\|^2_{L^2} + C_0,
    \end{equation}
    where $C_0=C_0(\|u_0\|_{L^2},p,N)$ is a positive constant.\\
    From \eqref{B-energy-est2} and \eqref{Ener}, we obtain
    \begin{equation}
        \label{B-energy-est3}
        \mathbf{E}(t) \ge \frac12 \|\nabla u(s)\|^2_{L^2} - C_0.
    \end{equation}
    Combining \eqref{B-energy-est1}--\eqref{B-energy-est3}, we conclude that
    \begin{equation}
        \label{B-energy-est4}
        \|\nabla u(t)\|^2_{L^2} \le 2 \|\nabla u_0\|^2_{L^2} + 2 C_0 +4 C_0 \baa(t)+ 2\int_0^t \ba(s) \|\nabla u(s)\|^2_{L^2} ds,    \end{equation}
where $\baa(t)$ is given by \eqref{A}.
    Thanks to Gr\"onwall's inequality, the estimate \eqref{B-energy-est4} yields
    \begin{equation}
        \label{B-energy-est5}
        \|\nabla u(t)\|^2_{L^2} \le \big(2 \|\nabla u_0\|^2_{L^2} + 2 C_0 +4 C_0 \baa(t)\big) {\rm e}^{2 \baa(t)}, \quad \forall \ 0 \le t < T_{max}.
    \end{equation}
    Since $\ba \in L^1_{loc}(0,\infty)$ and $T_{max} < \infty$, it follows from \eqref{B-energy-est5} that 
    $$\limsup_{t \to T_{max}}\|\nabla u(t)\|_{L^2}  < \infty.$$
    This ends the proof of Proposition \ref{prop-appendixB}.
\end{proof}

\section*{Acknowledgments}
The authors would like to thank the reviewer for the valuable comments and suggestions which improved the manuscript.









\begin{thebibliography}{99}



\bibitem{AT} L. Aloui and S. Tayachi, Local well-posedness for the inhomogeneous nonlinear Schr\"odinger equation, {\em Discrete Cont. Dyn. Syst.,} {\bf 41} (2021), 5409--5437.

\bibitem{Antonelli2} P. Antonelli, R. Carles and C. Sparber, On nonlinear Schr\"odinger-type equations with nonlinear damping, {\em Int. Math. Res. Not.,}  {\bf 3} (2015), 740--762.

\bibitem{Cazenave} {T. Cazenave}, Semilinear Schr\"odinger equations, {\em Courant Lecture Notes in Mathematics,} {\bf 10}. New York University, Courant Institute of Mathematical Sciences, AMS, 2003.

\bibitem{CPKP06} M. Centurion, M. A. Porter, P. G. Kevrekidis and D. Psaltis, Nonlinearity management in optics: experiment, theory, and simulation, {\em Phys. Rev. Lett.,} {\bf 97} (2006) 033903.

\bibitem{CZW} G. Chen, J. Zhang and Y. Wei, A small initial data criterion of global existence for the damped nonlinear Schr\"odinger equation, {\em J. Phys. A: Math. Theor.,} {\bf 42} (2009).

\bibitem{CIMM} J. Colliander, S. Ibrahim, M. Majdoub and N. Masmoudi, Energy critical NLS in two space dimensions, {\em J. Hyperbolic Differ. Equ.,} {\bf 6} (2009), 549--575.

\bibitem{Dar} M. Darwich, Blow-up for the damped $L^2-$critical nonlinear Schr\"odinger equation, {\em Adv. Differential Equations,} {\bf 17} (2012), 337--367.

\bibitem{Dar2} M. Darwich, {\em On the Cauchy problem for the nonlinear Schr\"odinger equations including fractional dissipation with variable coefficient,} 
 Math. Meth. Appl. Sci., {\bf 41} (2018), 2930--2938.

\bibitem{VDD} {V. D. Dinh}, Blow-up criteria for linearly damped nonlinear Schr\"odinger equations, {\em Evol. Equ. Control Theory,} {\bf 10} (2021), 599--617. 



\bibitem{DMS} V. D. Dinh, M. Majdoub and T. Saanouni, Long time dynamics and blow-up for the focusing inhomogeneous nonlinear Schr\"odinger equation with spatial growing nonlinearity, {\em J. Math. Phys.}, {\bf 64} (2023), 081509.

\bibitem{FZS14}{B. Feng, D. Zhao and C. Sun}, On the {Cauchy} problem for the nonlinear {Schr{\"o}dinger} equations with time-dependent linear loss/gain, {\em J. Math. Anal. Appl.}, {\bf 416} (2014), {901--923}.

\bibitem{Fib} {G. Fibich}, Self-focusing in the damped nonlinear Schr\"odinger equation, {\em SIAM J. Appl. Math.,} {\bf 61} (2001), 1680--1705.



  
 \bibitem{HH-MJM2022}{M. Hamouda and M.-A Hamza}, Improvement on the blow-up for the weakly coupled wave equations with scale-invariant damping and time derivative nonlinearity, {\em Mediterr. J. Math.}, {\bf 19} (2022), Pages {17}, {Id/No 136}.

 
\bibitem{In} T. Inui, Asymptotic behavior of the nonlinear damped Schr\"odinger equation, {\em Proc. Amer. Math. Soc.,} {\bf 147} (2019), 763--773. 

\bibitem{Kato} T. Kato, On nonlinear Schr\"odinger equations, {\em  Ann. Inst. H. Poincar\'e Phys. Th\'eor.,} {\bf 46} (1987), 113--129.


\bibitem{LST-2020} {N.-A. Lai, N. M. Schiavone and H. Takamura}, Heat-like and wave-like lifespan estimates for solutions of semilinear damped wave equations via a Kato's type lemma, {\em  J. Differential Equations,} Vol. {\bf 269}, (2020),  11575--11620. 

 \bibitem{Malomed} B. A. Malomed, Soliton Management in Periodic Systems, {\em Springer,} NewYork, 2006.

\bibitem{Merle} F. Merle and P. Rapha\"el, The blow-up dynamic and upper bound on the blow-up rate for critical nonlinear Schr\"odinger equation, {\em Ann. of Math.,} {\bf 161} (2005), 157--222.

\bibitem{Merle1} F. Merle and P. Rapha\"el, On a sharp lower bound on the blow-up rate for the $L^2-$critical nonlinear Schr\"odinger equation, {\em J. Amer. Math. Soc.,} {\bf 19} (2006), 37--90.
	
\bibitem{OT-1D} {T. Ogawa and Y. Tsutsumi}, Blow-up of $H^1-$solutions for the one dimensional nonlinear Schr\"odinger equation with critical power nonlinearity, {\em Proc. Amer. Math. Soc.,} {\bf 111} (1991), 487--496.

\bibitem{OhTo} {M. Ohta and G. Todorova}, Remarks on global existence and blow-up for damped nonlinear Schr\"odinger equations, {\em Discrete \& Continuous Dynamical Systems-A,} {\bf 23} (2009), 1313--1325.

\bibitem{Tarek} T. Saanouni, Remarks on the damped nonlinear Schr\"odinger equation, {\em Evol. Equ. Control Theory,} {\bf 9} (2020),  721--732.


\bibitem{Strauss} W. A. Strauss, Decay and asymptotics for {{\(\square u = F(u)\)}}, {\em J. Funct. Anal.}, {\bf 2} (1968), 409--457.


\bibitem{Tao} {T. Tao}, Nonlinear Dispersive Equations: Local and Global Analysis, {\em CBMS, Regional Conference Series in Mathematics} {\bf 106}, American Mathematical Society (2006).

\bibitem{Tao07}  T. Tao, M. Visan and X. Zhang, The nonlinear Schr\"odinger equation with combined power-type nonlinearities, {\em Comm. Partial Differential Equations,} {\bf 32} (2007), 1281--1343.

\bibitem{Tsut1} {M. Tsutsumi}, Nonexistence of global solutions to the Cauchy problem for the damped non-linear Schr\"odinger equations, {\em SIAM J. Math. Anal.,} {\bf 15} (1984), 357--366.


\end{thebibliography}
\end{document}